\documentclass{svproc}

\usepackage{url}

\usepackage{amssymb,amsmath,mathrsfs,mathtools}
\usepackage{bm,bbm}
\usepackage{hyperref}
\usepackage[capitalise]{cleveref}
\DeclareMathOperator{\re}{Re}

\begin{document}
\mainmatter              
\title{The Rate of Convergence for Selberg's Central Limit Theorem under the Riemann Hypothesis}
\titlerunning{The Rate of Convergence for Selberg's CLT under RH}  
%
\author{Asher Roberts}
\authorrunning{Asher Roberts} 
%
\tocauthor{Asher Roberts}
\institute{Graduate Center, City University of New York, New York, NY 10016, USA,\\
\email{aroberts@gradcenter.cuny.edu},\\ WWW home page:
\texttt{https://asherbroberts.com}
}

\maketitle              

\begin{abstract}
We assume the Riemann hypothesis to improve upon the rate of convergence of \((\log\log\log T)^2/\sqrt{\log\log T}\) in Selberg's central limit theorem for \(\log|\zeta(1/2+it)|\) given by the author in \cite{roberts}. We achieve a rate of convergence of \(\sqrt{\log\log\log\log T}/\sqrt{\log\log T}\) in the Dudley distance. The proof is an adaptation of the techniques used by the author in \cite{roberts}, based on the work of Radziwi\l\l\ and Soundararajan in \cite{rands} and Arguin et al. in \cite{maxzeta}, combined with a lemma of Selberg \cite{selbergmollifier} that provides for a mollifier close to the critical line \(\re(s)=1/2\) under the Riemann hypothesis.
\end{abstract}
\section{Introduction}

Let \(\tau\) be a random point uniformly distributed on \([T,2T]\). Selberg's central limit theorem states that as \(T\to\infty\), 

\[\mathbb{P}\left(\log|\zeta(\tfrac{1}{2}+i\tau)|> V\sqrt{\tfrac{1}{2}\log\log T}\right)\to\frac{1}{\sqrt{2\pi}}\int_V^\infty e^{-u^2/2}du,\quad \forall V\in\mathbb{R}.\]

In this paper, we determine the rate of convergence to a Gaussian random variable in Selberg's theorem using the Dudley (also called bounded Wasserstein, e.g. see Gaunt and Li \cite{dudley}) distance between two random vectors in \(\mathbb{R}^n\) defined on a probability space \((\Omega,\mathcal{F},\mathbb{P})\) given by \[d_{\mathcal{D}}(X,Y)=\sup_{f\in\mathcal{L}}|\mathbb{E}[f(X)]-\mathbb{E}[f(Y)]|\] where \(\mathcal{L}=\{f:\mathbb{R}^n\to\mathbb{R}\mid f\) is Lipschitz, \(\|f\|_\infty\leq 1\), and \(\|f\|_{\text{Lip}}\leq 1\}\). Here and throughout the paper \(\|f\|_\infty\) denotes the supremum norm of \(f\) and \(\|f\|_{\text{Lip}}=\sup_{x\ne y}\frac{|f(x)-f(y)|}{\|x-y\|_1}\) with \(\|\cdot\|_1\) denoting the \(L^1\) norm. We use the Dudley distance since the boundedness property allows us to use \cref{lma:fourierestimate} in \cref{prop:comparingmomentsrh}, while the Lipschitz property assists us with the remaining propositions in the paper.

\subsection{Main Result}

\begin{theorem}\label{thm:selbergcltraterh}
Assume the Riemann hypothesis and let \(\tau\) be a random point distributed uniformly on \([T,2T]\). Then for \(T\) large enough, \[d_{\mathcal{D}}\left(\frac{\log|\zeta(\tfrac{1}{2}+i\tau)|}{\sqrt{\tfrac{1}{2}\log\log T}},\mathcal{Z}\right)\ll\frac{\sqrt{\log\log\log\log T}}{\sqrt{\log\log T}},\]
where \(\mathcal{Z}\) is a standard normal random variable.
\end{theorem}

\cref{thm:selbergcltraterh} is an improvement of the rate of convergence \[d_{\mathcal{D}}\left(\frac{\log|\zeta(\tfrac{1}{2}+i\tau)|}{\sqrt{\tfrac{1}{2}\log\log T}},\mathcal{Z}\right)\ll\frac{(\log\log\log T)^2}{\sqrt{\log\log T}}\] given by the author in \cite{roberts} using the Dudley distance. We modify the techniques of the author in \cite{roberts} to prove \cref{thm:selbergcltraterh} by mollifying \(\log|\zeta(1/2+it)|\) closer to the critical line \(\re(s)=1/2\). This allows us to approximate \(\log|\zeta(1/2+it)|\) by an ``almost'' Gaussian prime sum using fewer steps, yielding a smaller error.

\subsection{Relations to Previous Results} In 1946 Selberg presented a proof of his central limit theorem for the Riemann zeta function in \cite{selberg}. Later, Radziwi\l\l\ and Soundararajan used simpler techniques to prove Selberg's theorem for \(\log|\zeta(\frac{1}{2}+it)|\) in \cite{rands}. It is challenging to apply these techniques to the imaginary case, due to the difficulty in mollifying the imaginary part of zeta. Nevertheless, the mollifier used in \cref{prop:mollifyingrh} of this paper can be applied to the imaginary part of zeta, but this would require a different way of moving the imaginary part of zeta off axis than we present in the proof of \cref{prop:movingoffaxis}. However, our work be extended to obtain analogous results for Dirichlet \(L\)-functions, as Hsu and Wong did in \cite{wongandhsu} to prove the central limit theorem using the techniques of Radziwi\l\l\ and Soundararajan.

Our paper mainly uses the techniques employed by the author in \cite{roberts}, although some modifications are needed. Here the rate of convergence was obtained for zeta at two points, while our result is presented for zeta at one point for simplicity but can be similarly extended. The main difference lies in the mollification step \cref{prop:mollifyingrh}. Without the Riemann hypothesis, the mollifier Selberg offers in \cite{selbergmollifier} presents a sum over zeta zeros that is difficult to estimate. By assuming the Riemann hypothesis and moving off axis by \(1/\log T\) (instead of by \((\log\log\log T)^2/\log T\) in \cite{roberts}), we can estimate this sum to be as close to order \(1\) as we would like by further iterating the log in the exponent of the parameter \(X\) used.

In \cite{selbergamalfi} Selberg claims a rate of convergence of \((\log\log\log T)^2/\sqrt{\log\log T}\) unconditionally for the real part of zeta using the Kolmogorov distance, a proof of which is given by Tsang in \cite{tsang}. The author reproduces this rate using the techniques of Radziwi\l\l\ and Soundararajan as well as Arguin et al. in \cite{roberts}. Selberg also claims a rate of \(\log\log\log T/\sqrt{\log\log T}\) for the imaginary case. An explicit proof of this result is given by Wahl in \cite{wahl} in the context of mod-gaussian convergence (more information is given in Kowalski and Nikeghbali \cite{modgaussian}).

While the Riemann hypothesis allows us to improve upon Selberg's rate of convergence in \cite{selbergamalfi}, there exist barriers towards further improvement. Specifically, although we can push the parameter \(X\) further in \cref{prop:mollifyingrh} as stated, we are still limited in the choice of parameters \(X_1\) and \(X_2\) used in the Dirichlet polynomials we compare to Gaussian distributions in \cref{prop:comparingmomentsrh}. That is, iterating the log further in those parameters results in Dirichlet polynomials that align less closely with Gaussian distributions, worsening the errors presented in \crefrange{lma:fourierexpectations}{lma:fourierexpectationsp2}. Improving our overall stated rate of convergence therefore requires more sophisticated techniques in handling parameters with log iterations that approach order \(1\).

\subsection{Structure of the Proofs}

The proof of \cref{thm:selbergcltraterh} follows by five successive propositions. All proofs are given in \cref{proofs}, and the distances between the random variables are explicit at each step of the proof. The trickiest step is \cref{prop:mollifyingrh}, mollifying zeta. In particular, handling the sum over zeta zeros. However, the step that incurs the greatest error is \cref{prop:truncatingprimesumrh}, truncating the prime sum. Although we could improve the error obtained in mollifying by iterating the logarithm further in the definition of the parameter \(X\) used in \cref{prop:mollifyingrh}, we would still be required to perform the same truncation in \cref{prop:truncatingprimesumrh}, in order to effectively compare moments with the Fourier transforms in \cref{prop:comparingmomentsrh}. That is, an improvement to the rate of convergence we obtain requires an improvement in the way we truncate using \cref{prop:truncatingprimesumrh}.\\

To connect zeta with the aforementioned ``almost'' Gaussian prime sum, we first move off axis using \cref{prop:movingoffaxis}, since the Riemann hypothesis will then ensure that we do not encounter a zero when mollifying zeta in \cref{prop:mollifyingrh}. To simplify notation, we set \begin{equation}\label{eq:sigma0ands}
\sigma_0=\frac{1}{2}+\frac{1}{\log T},\qquad s_0=\sigma_0+i\tau,\qquad\text{and}\qquad\mathfrak{s}=\sqrt{\tfrac{1}{2}\log\log T}.
\end{equation} The contribution towards the rate of convergence in \cref{thm:selbergcltraterh} by moving off axis is then given by

\begin{proposition}[Moving off axis]\label{prop:movingoffaxis}
With the notation above, we have
\[d_{\mathcal{D}}\left(\frac{\log|\zeta(\tfrac{1}{2}+i\tau)|}{\mathfrak{s}},\frac{\log|\zeta(s_0)|}{\mathfrak{s}}\right)\ll\frac{1}{\sqrt{\log\log T}}.\]
\end{proposition}

Off axis, we are now able to approximate zeta with a mollifier in the form of a Dirichlet polynomial \(\mathcal{P}(s)\) given in \cref{prop:mollifyingrh}. Note that \cref{prop:mollifyingrh} is the only proposition conditional on the Riemann hypothesis. The contribution towards the rate of convergence in \cref{thm:selbergcltraterh} by mollifying is then given by

\begin{proposition}[Mollifying]\label{prop:mollifyingrh}
Assume the Riemann hypothesis. Let \(X=T^{1/(\log_4T)^{1/4}}\) where \(\log_\ell\) stands for the logarithm iterated \(\ell\) times, and \[\Lambda_X(n)=\begin{cases} \Lambda(n) & \text{for }1\leq n\leq X,\\
\Lambda(n)\dfrac{\log\dfrac{X^2}{n}}{\log X} &\text{for }X\leq n\leq X^2.\end{cases}\] Set \[\mathcal{P}(s_0)=\sum_{n<X^2}\frac{\Lambda_X(n)}{n^{s_0}\log n}.\] Then \[d_{\mathcal{D}}\left(\frac{\log|\zeta(s_0)|}{\mathfrak{s}},\frac{\re\mathcal{P}(s_0)}{\mathfrak{s}}\right)\ll\frac{\sqrt{\log_4T}}{\sqrt{\log\log T}}.\]
\end{proposition}

We then discard the higher order primes in the sum \(\mathcal{P}(s)\), whose contribution towards the rate of convergence in \cref{thm:selbergcltraterh} is given by

\begin{proposition}[Discarding higher order primes]\label{prop:discardingprimes}
Set \[\mathcal{P}(s_0)=\sum_{n<X^2}\frac{\Lambda_X(n)}{n^{s_0}\log n}\qquad\text{and}\qquad P(s_0)=\re\sum_{p\leq X}\frac{1}{p^{s_0}}.\] With the notation above, we have
\[d_{\mathcal{D}}\left(\frac{\re\mathcal{P}(s_0)}{\mathfrak{s}},\frac{P(s_0)}{\mathfrak{s}}\right)\ll\frac{1}{\sqrt{\log\log T}}.\]
\end{proposition}

In order to compare moments in \cref{prop:comparingmomentsrh}, we truncate the prime sum. The more we truncate the sum, the faster the sum tends toward a Gaussian random variable. This truncation bears a contribution towards the rate of convergence in \cref{thm:selbergcltraterh} given by

\begin{proposition}[Truncating the prime sum]\label{prop:truncatingprimesumrh}
Set \begin{align*}
P_1(s_0)&=\re\sum_{p\leq X_1}\frac{1}{p^{s_0}},\\ P_2(s_0)&=\re\sum_{X_1< p\leq X_2}\frac{1}{p^{s_0}}\text{, and}\\
P_3(s_0)&=\re\sum_{X_2< p\leq X}\frac{1}{p^{s_0}},
\end{align*} so that \(P(s_0)=P_1(s_0)+P_2(s_0)+P_3(s_0)\) where \(X_1=T^{1/\log\log T}\) and \(X_2=T^{1/\log\log\log T}\). With the notation above, we have
\[d_{\mathcal{D}}\left(\frac{P(s_0)}{\mathfrak{s}},\frac{P_1(s_0)+P_2(s_0)}{\mathfrak{s}}\right)\ll\frac{\sqrt{\log_4T}}{\sqrt{\log\log T}}.\]
\end{proposition}

Finally, we compare our truncated sum to a Gaussian random variable, generating a contribution towards the rate of convergence in \cref{thm:selbergcltraterh} given by

\begin{proposition}\label{prop:comparingmomentsrh}
With the notation above, we have
\[d_{\mathcal{D}}\left(\frac{P_1(s_0)+P_2(s_0)}{\mathfrak{s}},\mathcal{Z}\right)\ll\frac{1}{\sqrt{\log\log T}},\] where \(\mathcal{Z}\) is a standard normal random variable.
\end{proposition}

\noindent {\bf Acknowledgements.} I give thanks to Emma Bailey for spotting a crucial issue in one of my initial drafts for this paper, and to Prof. Louis-Pierre Arguin for his helpful insight that provided me with the confidence to fix it. I would also like to thank everyone who read through the paper and offered edits. Your feedback is very much appreciated.\\
Partial support is provided by grants NSF CAREER 1653602 and NSF DMS 2153803.

\section{Proofs}\label{proofs}

\subsection{Proof of \cref{thm:selbergcltraterh}}\begin{proof}
With the notation above, we have by \crefrange{prop:movingoffaxis}{prop:comparingmomentsrh} and the triangle inequality, \begin{align*}
d_{\mathcal{D}}\left(\frac{\log|\zeta(\tfrac{1}{2}+i\tau)|}{\sqrt{\tfrac{1}{2}\log\log T}},\mathcal{Z}\right)&\ll d_{\mathcal{D}}\left(\frac{\log|\zeta(\tfrac{1}{2}+i\tau)|}{\mathfrak{s}},\frac{\log|\zeta(s_0)|}{\mathfrak{s}}\right)\\
&\phantom{\ll}+d_{\mathcal{D}}\left(\frac{\log|\zeta(s_0)|}{\mathfrak{s}},\frac{\re\mathcal{P}(s_0)}{\mathfrak{s}}\right)\\
&\phantom{\ll}+d_{\mathcal{D}}\left(\frac{\re\mathcal{P}(s_0)}{\mathfrak{s}},\frac{P(s_0)}{\mathfrak{s}}\right)\\
&\phantom{\ll}+d_{\mathcal{D}}\left(\frac{P(s_0)}{\mathfrak{s}},\frac{P_1(s_0)+P_2(s_0)}{\mathfrak{s}}\right)\\
&\phantom{\ll}+d_{\mathcal{D}}\left(\frac{P_1(s_0)+P_2(s_0)}{\mathfrak{s}},\mathcal{Z}\right)\\
&\ll\frac{\sqrt{\log\log\log\log T}}{\sqrt{\log\log T}}.
\end{align*}
\end{proof}

\subsection{Proof of \cref{prop:movingoffaxis}}\begin{proof}
Denote \(\log|\zeta(\tfrac{1}{2}+i\tau)|/\mathfrak{s}\) by \(\mathcal{V}\) and \(\log|\zeta(s_0)|/\mathfrak{s}\) by \(\mathcal{W}\). Then
\[d_{\mathcal{D}}(\mathcal{V},\mathcal{W})=\sup_{f\in\mathcal{L}}|\mathbb{E}[f(\mathcal{V})]-\mathbb{E}[f(\mathcal{W})]|\leq\sup_{f\in\mathcal{L}}\mathbb{E}[|f(\mathcal{V})-f(\mathcal{W})|].\] Since \(f:\mathbb{R}\to\mathbb{R}\) where \(f\in\mathcal{L}\) is Lipschitz with \(\|f\|_{\text{Lip}}\leq 1\), we have \(|f(v)-f(w)|\leq\|f\|_{\text{Lip}}\|v-w\|_1\leq\|v-w\|_1\), and thus \begin{equation}
|f(\mathcal{V})-f(\mathcal{W})|\leq\|\mathcal{V}-\mathcal{W}\|_1.\label{eq:lipschitzbound}
\end{equation} Therefore, by definition of the Dudley distance, the above distance is less than or equal to \(\mathbb{E}[|\mathcal{V}-\mathcal{W}|]\). We now introduce Proposition 1 of \cite{rands} as \cref{lma:offaxis}.

\begin{lemma}\label{lma:offaxis}
With the notation above, for \(T\) large enough and any \(\sigma>1/2\), we have \[\mathbb{E}\Big[\Big|\log|\zeta(\tfrac{1}{2}+i\tau)|-\log|\zeta(\sigma+i\tau)|\Big|\Big]\ll(\sigma-\tfrac{1}{2})\log T.\]
\end{lemma}

\noindent The proof of \cref{lma:offaxis} uses Hadamard's factorization formula for the completed \(\zeta\)-function and is given in \cite{rands}. By \cref{lma:offaxis}, \(\mathbb{E}[|\mathcal{V}-\mathcal{W}|]\ll 1/\sqrt{\log\log T}\), and so our result follows.
\end{proof}

\subsection{Proof of \cref{prop:mollifyingrh}}\begin{proof}
Similar to the way we proved \cref{prop:movingoffaxis}, we denote \(\log|\zeta(s_0)|/\mathfrak{s}\) by \(\mathcal{W}\), \(\re\mathcal{P}(s_0)/\mathfrak{s}\) by \(\mathcal{X}\) and have
\begin{equation}
d_{\mathcal{D}}(\mathcal{W},\mathcal{X})=\sup_{f\in\mathcal{L}}|\mathbb{E}[f(\mathcal{W})]-\mathbb{E}[f(\mathcal{X})]|\leq\sup_{f\in\mathcal{L}}\mathbb{E}[|f(\mathcal{W})-f(\mathcal{X})|].\label{eq:WXdistance}
\end{equation}

To compute this distance, we first approximate zeta using a mollifier of Selberg \cite{selbergmollifier} given by the following lemma.

\begin{lemma}\label{lma:selberg}
With the same assumptions as \cref{prop:mollifyingrh}, and for \(s\ne 1\), \(s\neq\rho\), and \(s\ne -2q\) \((q=1,2,3,\ldots)\), where \(\rho\) is a zero of \(\zeta\), we have \[\frac{\zeta'}{\zeta}(s)=-\sum_{n<X^2}\frac{\Lambda_X(n)}{n^s}+\frac{1}{\log X}\sum_\rho\frac{X^{\rho-s}-X^{2(\rho-s)}}{(s-\rho)^2}
+O(1).\]
\end{lemma}

\begin{proof}
The proof of the result \begin{align*}
\frac{\zeta'}{\zeta}(s)=-\sum_{n<X^2}\frac{\Lambda_X(n)}{n^s}&+\frac{X^{2(1-s)}-X^{1-s}}{\log X\cdot(1-s)^2}+\frac{1}{\log X}\sum_{q=1}^\infty\frac{X^{-2q-s}-X^{-2(2q+s)}}{(2q+s)^2}\\
&+\frac{1}{\log X}\sum_\rho\frac{X^{\rho-s}-X^{2(\rho-s)}}{(s-\rho)^2}
\end{align*} is given unconditionally in Lemma 2 of Selberg \cite{selbergmollifier} using a contour integral. For \(s=\sigma+it\) we have \(|1-s|^2\asymp t^2\), and so \[\frac{X^{2(1-s)}-X^{1-s}}{\log X\cdot(1-s)^2}\ll\frac{X^{2(1-\sigma)}+X^{1-\sigma}}{t^2\log X}\ll\frac{X}{t^2\log X}\ll\frac{1}{T\log T},\] while the sum over \(q\) is clearly \(O(1)\).
\end{proof}

We continue by applying the Riemann hypothesis, setting \(s=\sigma+it\) and \(\rho=1/2+i\gamma\), while integrating both sides of \cref{lma:selberg} to obtain
\[\int_{\sigma_0}^2\re\frac{\zeta'}{\zeta}(s)\,d\sigma=-\re\sum_{n<X^2}\Lambda_X(n)\int_{\sigma_0}^2n^{-s}d\sigma+\frac{1}{\log X}\re\sum_\rho\int_{\sigma_0}^2\frac{X^{\rho-s}-X^{2(\rho-s)}}{(s-\rho)^2}\,d\sigma.\] The first summand becomes 
\begin{align*}
-\re\sum_{n<X^2}\Lambda_X(n)\int_{\sigma_0}^2n^{-\sigma-it}d\sigma&=-\re\sum_{n<X^2}\Lambda_X(n)n^{-it}\left[\frac{\exp(-\sigma\log n)}{-\log n}\right]_{\sigma=\sigma_0}^{\sigma=2}\\
&=-\re\sum_{n<X^2}\Lambda_X(n)n^{-it}\left[\frac{n^{-\sigma}}{-\log n}\right]_{\sigma=\sigma_0}^{\sigma=2}\\
&=-\re\sum_{n<X^2}\Lambda_X(n)\left[\frac{n^{-s_0}}{\log n}-\frac{n^{-2-it}}{\log n}\right]\\
&=-\re\sum_{n<X^2}\frac{\Lambda_X(n)}{\log n}n^{-s_0},
\end{align*}
by the summability of \(\sum n^{-2-it}\). Meanwhile, the second summand is \begin{align*}
&\frac{1}{\log X}\re\sum_\rho\int_{\sigma_0}^2\frac{X^{\rho-s}-X^{2(\rho-s)}}{(s-\rho)^2}\,d\sigma\\
&\ll\frac{1}{\log X}\sum_\rho\frac{1}{|s_0-\rho|^2}\int_{\sigma_0}^2 (X^{1/2-\sigma}+X^{2(1/2-\sigma)})\,d\sigma\\
&\ll\frac{1}{(\log X)^2}\sum_\rho\frac{1}{|s_0-\rho|^2}.
\end{align*} Therefore, \[\log|\zeta(2+it)|-\log|\zeta(s_0)|=-\re\sum_{n<X^2}\frac{\Lambda_X(n)}{\log n}n^{-s_0}+O\left(\frac{1}{(\log X)^2}\sum_\rho\frac{1}{|s_0-\rho|^2}\right).\] As we obtained the inequality \eqref{eq:lipschitzbound} by taking advantage of the Lipschitz nature of \(f\), we similarly obtain \[|f(\mathcal{W})-f(\mathcal{X})|\leq\|\mathcal{W}-\mathcal{X}\|_1,\] and so the expectation in \eqref{eq:WXdistance} is less than or equal to \begin{align*}
\mathbb{E}[|\mathcal{W}-\mathcal{X}|]&=\frac{1}{\mathfrak{s}}\mathbb{E}[|\log|\zeta(s_0)|-\re\mathcal{P}(s_0)|]\\
&\ll\frac{1}{\mathfrak{s}}\mathbb{E}\left[\frac{1}{(\log X)^2}\sum_\rho\frac{1}{|s_0-\rho|^2}\right],
\end{align*} since \(\log|\zeta(2+it)|=O(1)\).
We then proceed similarly as in the proof of Proposition 1 of Harper \cite{harper}, computing the expectation by integrating
\begin{align*}
\int_{T}^{2T}\sum_{\rho:T\leq\gamma\leq 2T}\frac{1}{|s_0-\rho|^2}\,dt&=\sum_{\rho:T\leq\gamma\leq 2T}\int_{T}^{2T}\frac{1}{|s_0-\rho|^2}\,dt.
\end{align*}
Since we are assuming the Riemann hypothesis, and since we have moved off axis by \(1/\log T\), we have \(|s_0-\rho|>1/\log T\) for all \(\rho\), and therefore the integral above is \begin{align*}
\leq \sum_{\rho:T\leq\gamma\leq 2T}\int_{1/\log T}^\infty\frac{1}{t^2}\,dt=\log T\sum_{\rho:T\leq\gamma\leq 2T}1\ll(\log T)^2,
\end{align*}
where in the end we used the Riemann-von Mangoldt formula (see, e.g., Ivi\'c \cite{ivic}), which provides an estimate of \(\log T\) for the number of zeta zeros between \(T\) and \(2T\). Thus \[\frac{1}{\mathfrak{s}}\mathbb{E}[|\log|\zeta(s_0)|-\re\mathcal{P}(s_0)|]\ll\frac{1}{\mathfrak{s}}\cdot\frac{1}{(\log X)^2}\cdot(\log T)^2\ll\frac{\sqrt{\log_4T}}{\sqrt{\log\log T}}\] as desired.
\end{proof}

\subsection{Proof of \cref{prop:discardingprimes}}\begin{proof}
By taking advantage of the Lipschitz nature of the Dudley distance, we obtain \begin{align*}
&d_{\mathcal{D}}\left(\frac{\re\mathcal{P}(s_0)}{\mathfrak{s}},\frac{P(s_0)}{\mathfrak{s}}\right)\\
&\leq\mathbb{E}\left[\left|\frac{\re\mathcal{P}(s_0)}{\mathfrak{s}}-\frac{P(s_0)}{\mathfrak{s}}\right|\right].
\end{align*}
Then \(\Lambda_X(n)\leq\Lambda(n)\), and so \[\sum_{n<X^2}\frac{\Lambda_X(n)}{n^{s_0}\log n}\ll\sum_{n<X}\frac{\Lambda(n)}{n^{s_0}\log n},\] since \(T^{2/(\log_4T)^{1/4}}=O(T^{1/(\log_4T)^{1/4}})\).
Therefore, computing the difference in the expectation we get \begin{equation}
\frac{1}{\mathfrak{s}}|\re\mathcal{P}(s_0)-P(s_0)|\ll\frac{1}{\mathfrak{s}}\left|\sum_{2\leq p^2\leq X}\frac{1}{2p^{2s_0}}\right|+\frac{1}{\mathfrak{s}}\left|\sum_{\substack{2\leq p^k\leq X,\\k\geq 3}}\frac{\log p}{p^{ks_0}(k\log p)}\right|.\label{eq:reP-P}
\end{equation} Then since \[\left|\sum_{\substack{2\leq p^k\leq X,\\k\geq 3}}\frac{\log p}{p^{ks_0}(k\log p)}\right|\leq\sum_{\substack{2\leq p^k\leq X,\\k\geq 3}}\frac{1}{3p^{k\sigma_0}}=O(1)\]
and \begin{align*}
\mathbb{E}\left[\left|\sum_{2\leq p^2\leq X}\frac{1}{2p^{2(s_0)}}\right|\right]&\leq\mathbb{E}\left[\left|\sum_{2\leq p^2\leq X}\frac{1}{2p^{2(s_0)}}\right|^2\right]^{1/2}\\
&=\frac{1}{T}\int_T^{2T}\left|\sum_{2\leq p^2\leq X}\frac{1}{2p^{2(s_0)}}\right|^2dt\\
&=\frac{1}{T}\cdot\frac{1}{4}\sum_{p_1,p_2\leq\sqrt{X}}\int_T^{2T}\frac{1}{p_1^{2(\sigma_0+it)}p_2^{2(\sigma_0-it)}}\,dt\\
&\ll \sum_{p\leq\sqrt{X}}\frac{1}{p^{4\sigma_0}}+\frac{1}{T}\sum_{\substack{p_1,p_2\leq\sqrt{X}\\p_1\ne p_2}}\frac{1}{p_1^{2\sigma_0}p_2^{2\sigma_0}}\sqrt{p_1p_2}=O(1),
\end{align*}
we see that \eqref{eq:reP-P} is \(\ll 1/\mathfrak{s}\).
\end{proof}

\subsection{Proof of \cref{prop:truncatingprimesumrh}}\begin{proof}
We have
\begin{align*}
d_{\mathcal{D}}\left(\frac{P(s_0)}{\mathfrak{s}},\frac{P_1(s_0)+P_2(s_0)}{\mathfrak{s}}\right)\leq\mathbb{E}\left[\left(\frac{P_3(s_0)}{\mathfrak{s}}\right)^2\right]^{1/2}.
\end{align*}
To evaluate this expectation, we use Lemma 5 of \cite{roberts}, which we introduce as \cref{lma:moments}

\begin{lemma}\label{lma:moments}
For any non-negative integer \(k\) we have \begin{align*}
&\mathbb{E}\left[\left(\re\left(\sum_{p\leq X}\frac{1}{p^{s_0}}\right)\right)^{2k}\right]\\
&=2^{-k}\frac{(2k)!}{k!}\left[\left(\frac{1}{2}\sum_{p\leq X} \frac{1}{p^{2\sigma_0}}\right)^{k}+O_k\left(\frac{1}{2}\sum_{p\leq X} \frac{1}{p^{2\sigma_0}}\right)^{k-2}\right]+O\left(\frac{X^{4k}}{T}\right),
\end{align*}
where \(O_k\) indicates that the implied constant depends on \(k\).
If \(k\) is odd, we have \[\mathbb{E}\left[\left(\re\left(\sum_{p\leq X}\frac{1}{p^{s_0}}\right)\right)^k\right]=O\left(\frac{X^{2k}}{T}\right).\]
\end{lemma}

\begin{proof}
The statement follows from the proof of Lemma 5 of \cite{roberts} by taking \(u=1,u'=0\).
\end{proof}

By \cref{lma:moments}, we have \begin{equation}
\mathbb{E}[|P_3(\tfrac{1}{2}+it)|^{2k}]\ll\frac{(2k)!}{k!2^k}\left(\sqrt{\tfrac{1}{4}\log_4 T}\right)^{2k}+O\left(\frac{X^{4k}}{T}\right).\label{eq:P3moments}
\end{equation}
Here we used Mertens's theorem to estimate \[\sum_{X_2<p\leq X}\frac{1}{p}\ll\log\log X_2-\log\log X=\log(\sqrt{\log\log\log T})-\log((\log_4 T)^{1/4})\ll\tfrac{1}{2}\log_4 T.\] Thus the expectation above is less than or equal to \(\frac{1}{\mathfrak{s}}(\sqrt{\log_4 T}+O(1))\) when \(k=1\).
\end{proof}

\subsection{Proof of \cref{prop:comparingmomentsrh}}\begin{proof}

Similar to Proposotion 6 of the author in \cite{roberts}, we obtain the error in approximating \(P_1(s_0)+P_2(s_0)\) with a Gaussian random variable by comparing their Fourier transforms using Lemma 3.11 of \cite{maxzeta}, which we introduce as \cref{lma:fourierestimate},
\begin{lemma}\label{lma:fourierestimate}
Let \(\mu\) and \(\nu\) be two probability measures on \(\mathbb{R}\) with Fourier transform \(\hat{\mu}\) and \(\hat{\nu}\). There exists constant \(c>0\) such that for any function \(f:\mathbb{R}\to\mathbb{R}\) with Lipschitz constant \(1\) and for any \(R,F>0\), \[\left|\int_{\mathbb{R}}f\,d \mu-\int_{\mathbb{R}}f\,d \nu\right|\leq\frac{1}{F}+\|f\|_\infty\left\{(RF)\|(\hat{\mu}-\hat{\nu})\mathbbm{1}_{(-F,F)}\|_\infty+\mu((-R,R)^c)+\nu((-R,R)^c)\right\}.\]
\end{lemma}

To use \cref{lma:fourierestimate}, we introduce a lemma that shows that the difference in expectations between the Fourier transforms of the prime sum and that of a Gaussian random variable is small.

\begin{lemma}\label{lma:fourierexpectations}
\begin{align*}
\left|\mathbb{E}\left[\exp\left(\frac{i\xi P_1(s_0)}{\mathfrak{s}}\right)\right]-\mathbb{E}[\exp(i\xi \mathcal{Z})]\right|&\ll\frac{1}{(\log\log T)^2}.
\end{align*}
\end{lemma}

\begin{proof}
We first truncate the Fourier transform of \(P_1(s_0)\) for \(\xi\in\mathbb{R}\) by writing \begin{align*}
&\mathbb{E}\left[\exp\left(\frac{i\xi P_1(s_0)}{\mathfrak{s}}\right)\cdot\mathbbm{1}(|P_1(s_0)|\leq\log\log T)\right]\\
&\ll\sum_{n\leq N}\frac{1}{n!}i^n\mathbb{E}\left[\left(\frac{\xi P_1(s_0)}{\mathfrak{s}}\right)^n\cdot\mathbbm{1}(|P_1(s_0)|\leq\log\log T)\right]+(\log T)^{-C}
\end{align*}
for a constant \(C\), where the error term \((\log T)^{-C}\) is estimated in \cite{roberts} using Stirling's formula. The event \(\{|P_1(s_0)|\leq\log\log T\}\) is introduced to assist with the truncation, and is subsequently removed by inserting the complement of this event, \(\{|P_1(s_0)|>\log\log T\}\), and showing that the probability of the complement is small. After truncating, the moment estimate for \(P_1\) in \cref{lma:moments} is used to provide the error in comparing the truncated moments to a Gaussian. Full details of the proof are given in Lemma 9 of \cite{roberts}.
\end{proof}

Next we introduce \cref{lma:fourierexpectationsp2} as an analog of \cref{lma:fourierexpectations} for \(P_2(s_0)\).

\begin{lemma}\label{lma:fourierexpectationsp2}
\begin{align*}
\left|\mathbb{E}\left[\exp\left(\frac{i\xi P_2(s_0)}{\mathfrak{s}}\right)\right]-\mathbb{E}[\exp(i\xi \mathcal{Z})]\right|&\ll\frac{1}{(\log\log T)^2}.
\end{align*}
\end{lemma}

\begin{proof}
We truncate the Fourier transform of \(P_2(s_0)\) for \(\xi\in\mathbb{R}\) by writing \begin{align*}
&\mathbb{E}\left[\exp\left(\frac{i\xi P_2(s_0)}{\mathfrak{s}}\right)\cdot\mathbbm{1}(|P_2(s_0)|\leq 2\log\log\log T)\right]\\
&\ll\sum_{n\leq N}\frac{1}{n!}i^n\mathbb{E}\left[\left(\frac{\xi P_2(s_0)}{\mathfrak{s}}\right)^n\cdot\mathbbm{1}(|P_2(s_0)|\leq 2\log\log\log T)\right]+(\log T)^{-D}
\end{align*}
for a constant \(D\), similar to the way we truncated the Fourier transform of \(P_1(s_0)\). Note that we take \(2\log\log\log T\) in the event in order to achieve an error of \(1/(\log\log T)^2\), since
the proof requires an estimate for \(\mathbb{P}(|P_2(s_0)|>2\log\log\log T)\) instead of for \(\mathbb{P}(|P_1(s_0)|>\log\log T)\). As such,
\begin{align*}
\mathbb{P}(|P_2(s_0)|>2\log\log\log T)&\leq\frac{1}{(2\log\log\log T)^{2k}}\mathbb{E}\left[|P_2(s_0)|^{2k}\right]\\
&\ll\frac{1}{(2\log\log\log T)^{2k}}\left[\frac{(2k)!}{k!2^k}\left(\sqrt{\tfrac{1}{2}\log\log\log T}\right)^{2k}\right]\\
&\ll\frac{1}{(2\log\log\log T)^{2k}}\left[\frac{(2k)^{2k}}{k^k2^k}\frac{\exp(-2k)}{\exp(-k)}\left(\sqrt{\tfrac{1}{2}\log\log\log T}\right)^{2k}\right]\\
&=\frac{1}{(2\log\log\log T)^{2k}}\left[(2k)^{k}\exp(-k)2^{-k}(\log\log\log T)^{k}\right],
\end{align*}
so we can set \(k=\lfloor\log\log\log T\rfloor\) to get \begin{align*}
\mathbb{P}(|P_2(s_0)|>2\log\log\log T)&\ll\frac{1}{(2k)^{2k}}[(2k)^{k}\exp(-k)2^{-k}k^k]\\
&=2^{-2k}\exp(-k)\\
&\ll \exp(-2k)=\frac{1}{(\log\log T)^2}.
\end{align*} The rest of the proof is the same as for \cref{lma:fourierexpectations}. 
\end{proof}

Now, using \crefrange{lma:fourierestimate}{lma:fourierexpectationsp2} we have \begin{align*}
&d_{\mathcal{D}}\left[\left(\frac{P_1(s_0)+P_2(s_0)}{\mathfrak{s}}\right),\mathcal{Z}\right]\leq d_{\mathcal{D}}\left[\left(\frac{P_1(s_0)}{\mathfrak{s}}\right),\mathcal{Z}\right]+d_{\mathcal{D}}\left[\left(\frac{P_2(s_0)}{\mathfrak{s}}\right),\mathcal{Z}\right]\\
&\leq\frac{1}{F}+\left\{(R_1F)\left|\mathbb{E}\left[\exp\left(\frac{i\xi P_1(s_0)}{\mathfrak{s}}\right)\right]-\mathbb{E}[\exp(i\xi\mathcal{Z})]\right|\right.\\
&\phantom{\leq}\left.+\mathbb{P}\left(\left|\frac{P_1(s_0)}{\mathfrak{s}}\right|>R_1\right)+\mathbb{P}(|\mathcal{Z}|>R_1)\right\}\\
&\phantom{\leq}+\frac{1}{F}+\left\{(R_2F)\left|\mathbb{E}\left[\exp\left(\frac{i\xi P_2(s_0)}{\mathfrak{s}}\right)\right]-\mathbb{E}[\exp(i\xi\mathcal{Z})]\right|\right.\\
&\phantom{\leq}\left.+\mathbb{P}\left(\left|\frac{P_2(s_0)}{\mathfrak{s}}\right|>R_2\right)+\mathbb{P}(|\mathcal{Z}|>R_2)\right\}
\end{align*}
for any \(R_1,R_2,F>0\). Since \begin{align*}
&\mathbb{P}\left(|P_2(s_0)|>\sqrt{\log\log\log T}\sqrt{\log\log T}\right)
\\
&\leq\frac{1}{(\log\log\log T)(\log\log T)}\mathbb{E}[|P_2(s_0)|^2]=\frac{1}{\log\log T}
\end{align*} we can take \(R_1=\sqrt{\log\log T}\), \(R_2=\sqrt{\log\log\log T}\) and \(F=C\sqrt{\log\log T}\) to get \begin{align*}
&\frac{1}{\sqrt{\log\log T}}+\left[C\log\log T\left(\frac{1}{(\log\log T)^2}\right)+\frac{1}{\log T}+\frac{1}{(\log T)^D}\right]\\
&+\frac{1}{\sqrt{\log\log T}}+\left[C\sqrt{\log\log\log T}\sqrt{\log\log T}\left(\frac{1}{(\log\log T)^2}\right)+\frac{1}{\log\log T}+\frac{1}{(\log\log T)^D}\right]\\
&\ll\frac{1}{\sqrt{\log\log T}}.
\end{align*}
\end{proof}

%
%

\end{document}